\documentclass[12pt]{amsart}

\usepackage[left=1in, right=1in, top=1in, bottom=1in]
{geometry}
\usepackage{amsfonts,amssymb,graphicx,amsmath,amsthm, float}
\usepackage{xypic}

\theoremstyle{plain}
\newtheorem{theorem}{Theorem}[section]

\newtheorem{lemma}[theorem]{Lemma}
\newtheorem{example}[theorem]{Example}

\newtheorem{corollary}[theorem]{Corollary}
\newtheorem{question}[theorem]{Question}

\newcommand{\N}{\mathbb{N}}

\theoremstyle{remark}
\newtheorem*{remark}{Remark}

\newcommand{\Z}{\mathbb{Z}}

\begin{document}

\title{Generalized U-factorization in Commutative Rings with Zero-Divisors}
        \date{\today}

\author{Christopher Park Mooney}
\address{Reinhart Center \\ Viterbo University \\ 900 Viterbo Drive \\ La Crosse, WI 54601}
\email{cpmooney@viterbo.edu}



\begin{abstract}
Recently substantial progress has been made on generalized factorization techniques in integral domains, in particular $\tau$-factorization.  There has also been advances made in investigating factorization in commutative rings with zero-divisors.  One approach which has been found to be very successful is that of U-factorization introduced by C.R. Fletcher.  We seek to synthesize work done in these two areas by generalizing $\tau$-factorization to rings with zero-divisors by using the notion of U-factorization.  
\\
\vspace{.1in}\noindent \textbf{2010 AMS Subject Classification:} 13A05, 13E99, 13F15
\end{abstract}
\maketitle
\section{Introduction}  
Much work has been done on generalized factorization techniques in integral domains.  There is an excellent overview in \cite{Frazier}, where particular attention is paid to $\tau$-factorization.  Several authors have investigated ways to extend factorization to commutative rings with zero-divisors.  For instance, D.D. Anderson, Valdez-Leon, A\v{g}arg\"{u}n, Chun \cite{Chun, Valdezleon, Agargun}.  One particular method was that of U-factorization introduced by C.R. Fletcher in \cite{Fletcher} and \cite{Fletcher2}.  This method of factorization has been studied extensively by Michael Axtell and others in \cite{Axtell, Axtell2, Valdezleon3}.  We synthesize this work done into a single study of what we will call $\tau$-U-factorization. 
\\
\indent In this paper, we will assume $R$ is a commutative ring with $1$.   Let $R^*=R-\{0\}$, let $U(R)$ be the set of units of $R$, and let $R^{\#}=R^*-U(R)$ be the non-zero, non-units of $R$.  As in \cite{Axtell2}, we define U-factorization as follows.  Let $a\in R$ be a non-unit.  If $a=\lambda a_1\cdots a_n b_1\cdots b_m$ is a factorization with $\lambda \in U(R)$, $a_i,b_i \in R^{\#}$, then we will call $a=\lambda a_1 a_2 \cdots a_n \left\lceil b_1 b_2\cdots b_m \right\rceil$ a $\emph{U-factorization}$ of $a$ if (1) $\ a_i(b_1 \cdots b_m)=(b_1 \cdots b_m)$ for all $1\leq i \leq n$ and (2) $\ b_j(b_1 \cdots \widehat{b_j} \cdots b_m) \neq (b_1 \cdots \widehat{b_j} \cdots b_m)$ for $1 \leq j \leq m$ where $\widehat{b_j}$ means $b_j$ is omitted from the product.  Here $(b_1 \cdots b_m)$ is the principal ideal generated by $b_1 \cdots b_m$.  The $b_i$'s in this particular U-factorization above will be referred to as \emph{essential divisors}. The $a_i$'s in this particular U-factorization above will be referred to as \emph{inessential divisors}.  A U-factorization is said to be \emph{trivial} if there is only one essential divisor.
\\
\indent Note: we have added a single unit factor in front with the inessential divisors which was not in M. Axtell's original paper.  This is added for consistency with the $\tau$-factorization definitions and it is evident that a unit is always inessential.  We allow only one unit factor, so it will not affect any of the finite factorization properties.   

\begin{remark} If $a=\lambda a_1 \cdots a_n \left\lceil b_1 \cdots b_m\right\rceil$ is a U-factorization, then for any $1\leq i_0\leq m$, we have $(a)=(b_1 \cdots b_m) \subsetneq (b_1 \cdots \widehat{b_{i_0}} \cdots b_m)$.  This is immediate from the definition of U-factorization.
\end{remark}
\indent  In \cite{Axtell}, M. Axtell defines a non-unit $a$ and $b$ to be associate if $(a)=(b)$ and a non-zero non-unit $a$ said to be irreducible if $a=bc$ implies $a$ is associate to $b$ or $c$.  $R$ is  commutative ring $R$ to be \emph{U-atomic} if every non-zero non-unit has a U-factorization in which every essential divisor is irreducible.  $R$ is said to be a \emph{U-finite factorization ring} if every non-zero non-unit has a finite number of distinct U-factorizations.  $R$ is said to be a \emph{U-bounded factorization ring} if every non-zero non-unit has a bound on the number of essential divisors in any U-factorization.  $R$ is said to be a \emph{U-weak finite factorization ring} if every non-zero non-unit has a finite number of non-associate essential divisors.  $R$ is said to be a \emph{U-atomic idf-ring} if every non-zero non-unit has a finite number of non-associate irreducible essential divisors.  $R$ is said to be a \emph{U-half factorization ring} if $R$ is U-atomic and every U-atomic factorization has the same number of irreducible essential divisors.  $R$ is said to be a \emph{U-unique factorization ring} if it is a U-HFR and in addition each U-atomic factorization can be arranged so the essential divisors correspond up to associate.  In \cite[Theorem 2.1]{Axtell2}, it is shown this definition of U-UFR is equivalent to the one given by C.R. Fletcher in \cite{Fletcher, Fletcher2}.
\\
\indent In the second section, we begin with some preliminary definitions and results about $\tau$-factorization in integral domains as well as factorization in rings with zero-divisors.  In the third section, we state definitions for $\tau$-U-irreducible elements and $\tau$-U-finite factorization properties.  We also prove some preliminary results using these new definitions.  In the fourth section, we demonstrate the relationship between rings satisfying the various $\tau$-U finite factorization properties.  Furthermore, we compare these properties with the rings satisfying $\tau$-finite factorization properties studied in \cite{Mooney}.  In the final section, we investigate direct products of rings.  We introduce a relation $\tau_\times$ which carries many $\tau$-U-finite factorization properties of the component rings through the direct product.  

\section{Preliminary Definitions and Results} 
\indent As in \cite{Valdezleon}, we let $a \sim b$ if $(a)=(b)$, $a\approx b$ if there exists $\lambda \in U(R)$ such that $a=\lambda b$, and $a\cong b$ if (1) $a\sim b$ and (2) $a=b=0$ or if $a=rb$ for some $r\in R$ then $r\in U(R)$.  We say $a$ and $b$ are \emph{associates} (resp. \emph{strong associates, very strong associates}) if $a\sim b$ (resp. $a\approx b$, $a \cong b$).  As in \cite{Stickles}, a ring $R$ is said to be \emph{strongly associate} (resp. \emph{very strongly associate}) ring if for any $a,b \in R$, $a\sim b$ implies $a \approx b$ (resp. $a \cong b$).
\\
\indent Let $\tau$ be a relation on $R^{\#}$, that is, $\tau \subseteq R^{\#} \times R^{\#}$.  We will always assume further that $\tau$ is symmetric.  Let $a$ be a non-unit, $a_i \in R^{\#}$ and $\lambda \in U(R)$, then $a=\lambda a_1 \cdots a_n$ is said to be a \emph{$\tau$-factorization} if $a_i \tau a_j$ for all $i\neq j$.  If $n=1$, then this is said to be a \emph{trivial $\tau$-factorization}.  Each $a_i$ is said to be a $\tau$-factor, or that $a_i$ $\tau$-divides $a$, written $a_i \mid_\tau a$.
\\
\indent We say that $\tau$ is \emph{multiplicative} (resp. \emph{divisive}) if for $a,b,c \in R^{\#}$ (resp. $a,b,b' \in R^{\#}$), $a\tau b$ and $a\tau c$ imply $a\tau bc$ (resp. $a\tau b$ and $b'\mid b$ imply $a \tau b'$).  We say $\tau$ is \emph{associate} (resp. \emph{strongly associate}, \emph{very strongly associate) preserving} if for $a,b,b'\in R^{\#}$ with $b\sim b'$ (resp. $b\approx b'$, $b\cong b'$) $a\tau b$ implies $a\tau b'$.  We define a \emph{$\tau$-refinement} of a $\tau$-factorization $\lambda a_1 \cdots a_n$ to be a factorization of the form 
$$(\lambda \lambda_1 \cdots \lambda_n) \cdot b_{11}\cdots b_{1m_1}\cdot b_{21}\cdots b_{2m_2} \cdots b_{n1} \cdots b_{nm_n}$$
where $a_i=\lambda_i b_{i1}\cdots b_{im_i}$ is a $\tau$-factorization for each $i$.  This is slightly different from the original definition in \cite{Frazier} where no unit factor was allowed, and one can see they are equivalent when $\tau$ is associate preserving.  We then say that $\tau$ is \emph{refinable} if every $\tau$-refinement of a $\tau$-factorization is a $\tau$-factorization.  We say $\tau$ is \emph{combinable} if whenever $\lambda a_1 \cdots a_n$ is a $\tau$-factorization, then so is each $\lambda a_1 \cdots a_{i-1}(a_ia_{i+1})a_{i+2}\cdots a_n$.  
\\
\indent We now summarize several of the definitions given in \cite{Mooney}.  Let $a\in R$ be a non-unit.  Then $a$ is said to be \emph{$\tau$-irreducible} or \emph{$\tau$-atomic} if for any $\tau$-factorization $a=\lambda a_1 \cdots a_n$, we have $a\sim a_i$ for some $i$.  We will say $a$ is \emph{$\tau$-strongly irreducible} or \emph{$\tau$-strongly atomic} if for any $\tau$-factorization $a=\lambda a_1 \cdots a_n$, we have $a \approx a_i$ for some $a_i$.  We will say that $a$ is \emph{$\tau$-m-irreducible} or \emph{$\tau$-m-atomic} if for any $\tau$-factorization $a=\lambda a_1 \cdots a_n$, we have $a \sim a_i$ for all $i$.  Note: the $m$ is for ``maximal" since such an $a$ is maximal among principal ideals generated by elements which occur as $\tau$-factors of $a$.  We will say that $a$ is \emph{$\tau$-very strongly irreducible} or \emph{$\tau$-very strongly atomic} if $a\cong a$ and $a$ has no non-trivial $\tau$-factorizations.  See \cite{Mooney} for more equivalent definitions of these various forms of $\tau$-irreducibility.
\\
\indent From \cite[Theorem 3.9]{Mooney}, we have the following relations where $\dagger$ represents the implication requires a strongly associate ring:

$$
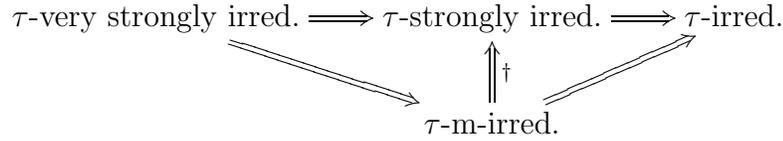
\xymatrix{
\tau\text{-very strongly irred.}\ar@{=>}[dr] \ar@{=>}[r]& \tau\text{-strongly irred.} \ar@{=>}[r]& \tau \text{-irred.}\\
 & \tau\text{-m-irred.}\ar@{=>}[u]_{\dagger}\ar@{=>}[ur]  & &}$$

\section{$\tau$-U-irreducible elements}
\indent  A \emph{$\tau$-U-factorization} of a non-unit $a\in R$ is a U-factorization  $a=\lambda a_1 a_2 \cdots a_n \left\lceil b_1 b_2\cdots b_m \right\rceil$ for which $\lambda a_1 \cdots a_n b_1 \cdots b_m$ is also a $\tau$-factorization.  
\\
\indent Given a symmetric relation $\tau$ on $R^{\#}$, we say $R$ is \emph{$\tau$-U-refinable} if for every $\tau$-U-factorization of any non-unit $a\in U(R)$, $a=\lambda a_1 \cdots a_n \left\lceil b_1 \cdots b_m\right\rceil$, any $\tau$-U-factorization of an essential divisors, $b_i=\lambda' c_1 \cdots c_{n'} \left\lceil d_1 \cdots d_{m'}\right\rceil$ satisfies 
$$a=\lambda \lambda' a_1 \cdots a_n c_1 \cdots c_{n'}\left\lceil b_1 \cdots b_{i-1} d_1 \cdots d_{m'} b_{i+1} \cdots \ b_m\right\rceil$$ is a $\tau$-U-factorization.  

\begin{example} Let $R=\Z/20\Z$, and let $\tau= R^{\#}\times R^{\#}$.  \end{example}
Certainly $0=\left\lceil 10\cdot 10\right\rceil$ is a $\tau$-U-factorization.  But $10=\left\lceil 2 \cdot 5\right\rceil$ is a $\tau$-U-factorization; however, $0=\left \lceil 2 \cdot 5 \cdot 2 \cdot 6 \right \rceil$ is not a U-factorization since $5$ becomes inessential after a $\tau$-U-refinement.  It will sometimes be important to ensure the essential divisors of a $\tau$-U-refinement of a $\tau$-U-factorization's essential divisors remain essential.  We will see that in a pr\'esimplifiable ring, there are no inessential divisors, so for $\tau$-refinable, $R$ will be $\tau$-U-refinable.
\\
\indent As stated in \cite{Axtell}, the primary benefit of looking at U-factorizations is the elimination of troublesome idempotent elements that ruin many of the finite factorization properties.  For instance, even $\Z_6$ is not a BFR (a ring in which every non-unit has a bound on the number of non-unit factors in any factorization) because we have $3=3^2$.  Thus, $3$ is an idempotent, so $3=3^n$ for all $n\geq 1$ which yields arbitrarily long factorizations.  When we use U-factorization, we see any of these factorizations can be rearranged to $3=3^{n-1}\left\lceil 3\right\rceil$, which has only one essential divisor.   
\\
\indent Let $\alpha \in \{$irreducible, strongly irreducible, m-irreducible, very strongly irreducible$\}$.  Let $a$ be a non-unit.  If $a=\lambda a_1 a_2 \cdots a_n \left\lceil b_1 b_2\cdots b_m \right\rceil$ is a $\tau$-U-factorization, then this factorization is said to be a \emph{$\tau$-U-$\alpha$-factorization} if it is a $\tau$-U-factorization and the essential divisors $b_i$ are $\tau$-$\alpha$ for $1 \leq i \leq m$.  
\\
\indent One must be somewhat more careful with U-factorizations as there is a loss of uniqueness in the factorizations.  For instance, if we let $R=\Z_6 \times \Z_8$, then we can factor $(3,4)$ as $(3,1) \left\lceil (3,3)(1,4)\right\rceil$ or $(3,3) \left\lceil (3,1)(1,4)\right\rceil$.  On the bright side, we have \cite[Proposition 4.1]{Valdezleon3}.

\begin{theorem} \label{thm: rearrange} Every factorization can be rearranged into a U-factorization.
\end{theorem}

\begin{corollary}\label{cor: rearrange}Let $R$ be a commutative ring with $1$ and $\tau$ a symmetric relation on $R^{\#}$.  Let $\alpha \in \{$irreducible, strongly irreducible, m-irreducible, very strongly irreducible$\}$.  For every $\tau$-$\alpha$ factorization of a non-unit $a\in R$, $a=\lambda a_1 \cdots a_n$, we can rearrange this factorization into a $\tau$-U-$\alpha$-factorization.
\end{corollary}

\begin{proof} Let $a=\lambda a_1 \cdots a_n$ be a $\tau$-$\alpha$-factorization.  By Theorem \ref{thm: rearrange} we can rearrange this to form a U-factorization.  This remains a $\tau$-factorization since $\tau$ is assumed to be symmetric.  Lastly each $a_i$ is $\tau$-$\alpha$, so the essential divisors are $\tau$-$\alpha$.
\end{proof}
\indent This leads us to another equivalent definition of $\tau$-irreducible. 
\begin{theorem} Let $a\in R$ be a non-unit.  Then $a$ is $\tau$-irreducible if and only if any $\tau$-U-factorization of $a$ has only one essential divisor.
\end{theorem}

\begin{proof} ($\Rightarrow$) Let $a$ be $\tau$-irreducible.  Let $a=\lambda a_1 \cdots a_n \left\lceil b_1 \cdots b_m\right\rceil$ be a $\tau$-U-factorization.  Suppose $m\geq 2$, then $a = \lambda a_1 \cdots a_n b_1 \cdots b_m$ is a $\tau$-factorization implies $a \sim a_{i_0}$ for some $1 \leq i_0 \leq n$ or $a \sim b_{i_0}$ for some $1\leq i_0 \leq m$.  But then either

$$(a)=(a_1 \cdots a_n b_1 \cdots b_m)\subsetneq (a_1 \cdots a_n \widehat{b_1}b_2\cdots b_{m}) \subseteq (a_{i_0}) = (a)$$ 
or
$$(a)=(a_1 \cdots a_n b_1 \cdots b_m)=(b_1 \cdots b_m)\subsetneq (\widehat{b_1} \cdots \widehat{b_{i_0 - 1}}\cdot b_{i_0}\cdot \widehat{b_{i_0 + 1}} \cdots \widehat{b_{m}}) \subseteq (b_{i_0}) = (a)$$
a contradiction.
\\
\indent ($\Leftarrow$) Suppose $a=\lambda a_1 \cdots a_n$.  Then this can be rearranged into a U-factorization, and hence a $\tau$-U-factorization.  By hypothesis, there can only be one essential divisor.  Suppose it is $a_n$.  We have $a=\lambda a_1 \cdots a_{n-1} \left\lceil a_n\right\rceil$ is a $\tau$-U-factorization and $a \sim a_n$ as desired.
\end{proof}
\indent We now define the finite factorization properties using the $\tau$-U-factorization approach.  Let $\alpha \in \{$ irreducible, strongly irreducible, m-irreducible, very strongly irreducible $\}$ and let $\beta \in \{$associate, strongly associate, very strongly associate $\}$.  $R$ is said to be \emph{$\tau$-U-$\alpha$} if for all non-units $a\in R$, there is a $\tau$-U-$\alpha$-factorization of $a$.  $R$ is said to satisfy \emph{$\tau$-U-ACCP} (ascending chain condition on principal ideals) if every properly ascending chain of principal ideals $(a_1) \subsetneq(a_2) \subsetneq \cdots $ such that $a_{i+1}$ is an essential divisor in some $\tau$-U-factorization of $a_i$, for each $i$ terminates after finitely many principal ideals.  $R$ is said to be a \emph{$\tau$-U-BFR} if for all non-units $a\in R$, there is a bound on the number of essential divisors in any $\tau$-U-factorization of $a$. 
\\
\indent $R$ is said to be a \emph{$\tau$-U-$\beta$-FFR} if for all non-units $a\in R$, there are only finitely many $\tau$-U-factorizations up to rearrangement of the essential divisors and $\beta$.  $R$ is said to be a \emph{$\tau$-U-$\beta$-WFFR} if for all non-units $a\in R$, there are only finitely many essential divisors among all $\tau$-U-factorizations of $a$ up to $\beta$.  $R$ is said to be a \emph{$\tau$-U-$\alpha$-$\beta$-divisor finite (df) ring} if for all non-units $a\in R$, there are only finitely many essential $\tau$-$\alpha$ divisors up to $\beta$ in the $\tau$-U-factorizations of $a$. 
\\
\indent $R$ is said to be a \emph{$\tau$-U-$\alpha$-HFR} if $R$ is $\tau$-U-$\alpha$ and for all non-units $a\in R$, the number of essential divisors in any $\tau$-U-$\alpha$-factorization of $a$ is the same.  $R$ is said to be a \emph{$\tau$-U-$\alpha$-$\beta$-UFR} if $R$ is a $\tau$-U-$\alpha$-HFR and the essential divisors of any two $\tau$-U-$\alpha$-factorizations can be rearranged to match up to $\beta$.  
\\
\indent $R$ is said to be \emph{pr\'esimplifiable} if for every $x \in R$, $x=xy$ implies $x=0$ or $y\in U(R)$.  This is a condition which has been well studied and is satisfied by any domain or local ring.  We introduce two slight modifications of this.  $R$ is said to be $\tau$\emph{-pr\'esimplifiable} if for every $x \in R$, the only $\tau$-factorizations of $x$ which contain $x$ as a $\tau$-factor are of the form $x=\lambda x$ for a unit $\lambda$.  $R$ is said to be \emph{$\tau$-U-pr\'esimplifiable} if for every non-zero non-unit $x \in R$, all $\tau$-U-factorizations have no non-unit inessential divisors.
\begin{theorem} Let $R$ be a commutative ring with $1$ and let $\tau$ be a symmetric relation on $R^{\#}$.  We have the following.
\\
(1) If $R$ is pr\'esimplifiable, then $R$ is $\tau$-U-pr\'esimplifiable.  
\\
(2) If $R$ is $\tau$-U-pr\'esimplifiable, then $R$ is $R$ is $\tau$-pr\'esimplifiable.
\\
That is $\text{pr\'esimplifiable} \Rightarrow \tau\text{-U-pr\'esimplifiable} \Rightarrow \tau\text{-pr\'esimplifiable.}$  If $\tau= R^{\#}\times R^{\#}$, then all are equivalent.
\end{theorem}
\begin{proof} (1) Let $R$ be pr\'esimplifiable, and $x\in R^{\#}$.  Suppose $x=\lambda a_1 \cdots a_n \left \lceil b_1 \cdots b_m \right \rceil$ is a $\tau$-U-factorization.  Then $(x)=(b_1 \cdots b_m)$.  $R$ pr\'esimplifiable implies that all the associate relations coincide, so in fact $x \cong b_1 \cdots b_m$ implies that $\lambda a_1 \cdots a_n \in U(R)$ and hence all inessential divisors are units.
\\
\indent (2) Let $R$ be $\tau$-U-pr\'esimplifiable, and $x \in R$ such that $x=\lambda xa_1 \cdots a_n$ is a $\tau$-factorization.  We claim that $x= \lambda a_1 \cdots a_n \left \lceil x \right \rceil$ is a $\tau$-U-factorization.  For any $1\leq i \leq n$, $x \mid a_ix$ and $(a_i x)(\lambda a_1 \cdots \widehat{a_i} \cdots a_n)=x$ shows $a_i x \mid x$, proving the claim.  This implies $\lambda a_1 \cdots a_n \in U(R)$ as desired.
\\
\indent Let $\tau=R^{\#} \times R^{\#}$ and suppose $R$ is $\tau$-pr\'esimplifiable.  Suppose $x=xy$, for $x\neq 0$, we show $y\in U(R)$.  If $x\in U(R)$, then multiplying through by $x^{-1}$ yields $1=x^{-1}x=x^{-1}xy=y$ and $y\in U(R)$ as desired.  We may now assume $x\in R^{\#}$.  If $y=0$, then $x=0$, a contradiction.  If $y\in U(R)$ we are already done, so we may assume $y\in R^{\#}$.  Thus $x\tau y$, and $x=xy$ is a $\tau$-factorization, so $y\in U(R)$ as desired.
\end{proof}

\section{$\tau$-U-finite factorization relations}
We now would like to show the relationship between rings with various $\tau$-U-$\alpha$-finite factorization properties as well as compare these rings with the $\tau$-$\alpha$-finite factorization properties of \cite{Mooney}.
\begin{theorem} Let $R$ be a commutative ring with $1$ and let $\tau$ be a symmetric relation on $R^{\#}$.  Consider the following statements.
\\
(1) $R$ is a $\tau$-BFR.
\\
(2) $R$ is $\tau$-pr\'esimplifiable and for every non-unit $a_1 \in R$, there is a fixed bound on the length of chains of principal ideals $(a_i)$ ascending from $a_1$ such that at each stage $a_{i+1} \mid_{\tau} a_i$.
\\
(3) $R$ is $\tau$-pr\'esimplifiable and a $\tau$-U-BFR.
\\
(4) For every non-unit $a\in R$ , there are natural numbers $N_1(a)$ and $N_2(a)$ such that if $a=\lambda a_1 \cdots a_n \left\lceil b_1 \cdots b_m\right\rceil$ is a $\tau$-U-factorization, then $n \leq N_1(a)$ and $m \leq N_2(a)$.
\\
\indent Then (4) $\Rightarrow$ (1) and (2) $\Rightarrow$ (3).  For $\tau$ refinable, (1) $\Rightarrow$ (2) and for $R$ $\tau$-U-pr\'esimplifiable, (3) $\Rightarrow$ (4).  Thus all are equivalent if $R$ is $\tau$-U-pr\'esimplifiable and $\tau$ is refinable.  
\\
\indent Let $\star$ represent $\tau$ being refinable, and $\dagger$ represent $R$ being $\tau$-U-pr\'esimplifiable, then the following diagram summarizes the theorem.
$$\xymatrix{(1) \ar@{=>}[r]^{\star}& (2) \ar@{=>}[d]\\
(4)\ar@{=>}[u] & (3) \ar@{=>}[l]^{\dagger}}$$

\end{theorem}
\begin{proof} (1) $\Rightarrow$ (2) Let $\tau$ be refinable.  Suppose there were a non-trivial $\tau$-factorization $x=\lambda x a_1 \cdots a_n$ with $n \geq 1$.  Since $\tau$ is assumed to be refinable we can continue to replace the $\tau$-factor $x$ with this factorization.
$$x=\lambda x a_1 \cdots a_n=(\lambda \lambda) x a_1 \cdots a_na_1 \cdots a_n=\cdots = (\lambda \lambda \lambda) x a_1 \cdots a_n a_1 \cdots a_n a_1 \cdots a_n= \cdots$$
yields an unbounded series of $\tau$-factorizations of increasing length.  
\\
\indent Let $a_1$ be a non-unit in $R$.  Suppose $N$ is the bound on the length of any $\tau$-factorization of $a_1$.  We claim that $N$ satisfies the requirement of (2).  Let $(a_1) \subsetneq (a_2) \subsetneq \cdots$ be an ascending chain of principal ideals generated by elements which satisfy $a_{i+1} \mid_{\tau} a_i$ for each $i$.  Say $a_i=\lambda_i a_{i+1}a_{i1} \cdots a_{i{n_i}}$ for each $i$.  Furthermore, we can assume $n_i\geq 1$ for each $i$ or else the containment would not be proper.  Then we can write
$$a_1=\lambda_1 a_2a_{11} \cdots a_{1{n_1}}=\lambda_1 \lambda_2 a_3a_{21} \cdots a_{2{n_2}}a_{11} \cdots a_{1{n_1}}= \cdots.$$
\indent Each remains a $\tau$-factorization since $\tau$ is refinable and we have added at least one factor at each step.  If the chain were greater than length $N$ we would contradict $R$ being a $\tau$-BFR.
\\
\indent (2) $\Rightarrow$ (3) Let $a\in R$ be a non-unit.  Let $N$ be the bound on the length of any properly ascending chain of principle ideals ascending from $a$ such that $a_{i+1} \mid_{\tau} a_i$.  If $a=\lambda a_1 \cdots a_n \left\lceil b_1 \cdots b_m\right\rceil$ is a $\tau$-U-factorization, then we get an ascending chain with $b_1 \cdots b_{i-1} \mid_\tau b_1 \cdots b_i$ for each $i$: 
$$(a)=(b_1 \cdots b_m) \subsetneq (b_1 \cdots b_{m-1}) \subsetneq (b_1 \cdots b_{m-2}) \subsetneq \cdots \subsetneq (b_1b_2) \subsetneq (b_1).$$
Hence, $m \leq N$ and we have found a bound on the number of essential divisors in any $\tau$-U-factorization of $a$, making $R$ a $\tau$-U-BFR.
\\
\indent (3) $\Rightarrow$ (4) Let $a\in R$ be a non-unit.  Let $N_e(a)$ be the bound on the number of essential divisors in any $\tau$-U-factorization of $a$.  Since $R$ is $\tau$-U-pr\'esimplifiable, there are no inessential $\tau$-U-divisors of $a$.  We can set $N_1(a)=0$, and $N_2(a)=N_e(a)$ and see that this satisfies the requirements of the theorem.
\\
\indent (4) $\Rightarrow$ (1) Let $a\in R$ be a non-unit.  Then any $\tau$-factorization $a=\lambda a_1 \cdots a_n$ can be rearranged into a $\tau$-U-factorization, say $a=\lambda a_{s_1} \cdots a_{s_i} \left\lceil a_{s_{i+1}} \cdots a_{s_n}\right\rceil$.  But then $n=i+(n-i) \leq N_1(a) + N_2(a)$.  Hence the length of any $\tau$-factorization must be less than $N_1(a) + N_2(a)$ proving $R$ is a $\tau$-BFR as desired.
\end{proof}
\indent The way we have defined our finite factorization properties on only the essential divisors causes a slight problem.  Given a $\tau$-U-factorization $a=\lambda a_1 \cdots a_n \left \lceil b_1 \cdots b_m \right \rceil$, we only know that $a\sim b_1 \cdots b_m$.  This may no longer be a $\tau$-factorization of $a$, but rather only some associate of $a$.  This is easily remedied by insisting that our rings are strongly associate.  

\begin{lemma}\label{lem: sa} Let $R$ be a strongly associate ring with $\tau$ a symmetric relation on $R^{\#}$, and let $\alpha \in \{$irreducible, strongly irreducible, m-irreducible, very strongly irreducible$\}$. Let $a\in R$, a non-unit.  If $a=\lambda a_1 a_2 \cdots a_n \left\lceil b_1 b_2\cdots b_m \right\rceil$ is a $\tau$-U-$\alpha$-factorization, then there is a unit $\mu\in U(R)$ such that $a=\mu b_1 \cdots b_m$ is a $\tau$-$\alpha$-factorization.
\end{lemma}
\begin{proof} Let $a=\lambda a_1 a_2 \cdots a_n \left\lceil b_1 b_2\cdots b_m \right\rceil$ be a $\tau$-U-$\alpha$-factorization.  By definition, $(a)=(b_1 \cdots b_m)$, and $R$ strongly associate implies that $a \approx b_1 \cdots b_m$.  Let $\mu\in U(R)$ be such that $a=\mu b_1 \cdots b_m$.  We still have $b_i \tau b_j$ for all $i\neq j$, and $b_i$ is $\tau$-$\alpha$ for every $i$.  Hence $a=\mu b_1 \cdots b_m$ is the desired $\tau$-factorization, proving the lemma.
\end{proof}
\begin{theorem}\label{thm: gen rel} Let $R$ be a commutative ring with $1$, and let $\tau$ be a symmetric relation on $R^{\#}$.  Let $\alpha \in \{$irreducible, strongly irreducible, m-irreducible, very strongly irreducible$\}$, and $\beta \in \{$associate, strongly associate, very strongly associate $\}$.    We have the following.
\\
(1) If $R$ is $\tau$-$\alpha$, then $R$ is $\tau$-U-$\alpha$.
\\
(2) If $R$ satisfies $\tau$-ACCP, then $R$ satisfies $\tau$-U-ACCP.
\\
(3) If $R$ is a $\tau$-BFR, then $R$ is a $\tau$-U-BFR.
\\
(4) If $R$ is a $\tau$-$\beta$-FFR, then $R$ is a $\tau$-U-$\beta$-FFR.
\\
(5) Let $R$ be a $\tau$-$\beta$-WFFR, then $R$ is a $\tau$-U-$\beta$-WFFR.
\\
(6) Let $R$ be a $\tau$-$\alpha$-$\beta$-divisor finite ring, then $R$ is $\tau$-U-$\alpha$-$\beta$-divisor finite ring.
\\
(7) Let $R$ be a strongly associate $\tau$-$\alpha$-HFR (resp. $\tau$-$\alpha$-$\beta$-UFR), then $R$ is $\tau$-U-$\alpha$-HFR (resp. $\tau$-U-$\alpha$-$\beta$-UFR).
\end{theorem}
\begin{proof} (1) This is immediate from Corollary \ref{cor: rearrange}.
\\
\indent (2) Suppose there were a infinite properly ascending chain of principal ideals $(a_1) \subsetneq (a_2) \subsetneq \cdots$ such that $a_{i+1}$ is an essential divisor in some $\tau$-U-factorization of $a_i$, for each $i$.  Every essential $\tau$-U-divisor is certainly a $\tau$-divisor.  This would contradict the fact that $R$ satisfies $\tau$-ACCP.
\\
\indent (3) We suppose that there is a non-unit $a\in R$ with $\tau$-U-factorizations having arbitrarily large numbers of essential $\tau$-U-divisors.  Each is certainly a $\tau$-factorization, having at least as many $\tau$-factors as there are essential $\tau$-divisors, so this would contradict the hypothesis.
\\
\indent (4) Every $\tau$-U-factorization is certainly among the $\tau$-factorizations.  If the latter is finite, then so is the former.
\\
\indent (5) For any given non-unit $a\in R$, every essential $\tau$-U-divisor of $a$ is certainly a $\tau$-factor of $a$ which has only finitely many up to $\beta$.  Hence there can be only finitely many essential $\tau$-U-factors up to $\beta$.
\\
\indent (6) Let $a\in R$ be a non-unit.  Every essential $\tau$-U-$\alpha$-divisor of $a$ is a $\tau$-$\alpha$-factor of $a$.  There are only finitely many $\tau$-$\alpha$-divisors up to $\beta$, so then there can be only finitely many $\tau$-U-$\alpha$-divisors of $a$ up to $\beta$.
\\
\indent (7) We have already seen that $R$ being $\tau$-$\alpha$ implies $R$ is $\tau$-U-$\alpha$.  Let $a\in R$ be a non-unit.  We suppose for a moment there are two $\tau$-$\alpha$-U-factorizations: 
$$a=\lambda a_1 \cdots a_n \left\lceil b_1 \cdots b_m\right\rceil=\lambda' a'_1 \cdots a'_{n'} \left\lceil b'_1 \cdots b'_{m'}\right\rceil$$
such that $m \neq m'$ (resp. $m \neq m'$ or there is no rearrangement such that $b_i$ and $b'_i$ are $\beta$ for each $i$).  Lemma \ref{lem: sa} implies $\exists \mu, \mu' \in U(R)$ with $a=\mu b_1 \cdots b_m=\mu' b'_1 \cdots b'_{m'}$ are two $\tau$-$\alpha$-factorizations of $a$, so $m=m'$ (resp. $m=m'$ and there is a rearrangement so that $b_i$ and $b'_i$ are $\beta$ for each $1 \leq i \leq m$), a contradiction, proving $R$ is indeed a $\tau$-U-$\alpha$-HFR (resp. -$\beta$-UFR) as desired.
\end{proof}

\begin{theorem} \label{thm: tau u relations} Let $R$ be a commutative ring with 1 and $\tau$ a symmetric relation on $R^{\#}$.  Let $\alpha \in \{$irreducible, strongly irreducible, m-irreducible, very strongly irreducible$\}$, and let $\beta \in \{$associate, strongly associate, very strongly associate$\}$.  \\
(1) If $R$ is a $\tau$-U-$\alpha$-$\beta$-UFR, then $R$ is a $\tau$-$\alpha$-U-HFR. 
\\
(2) If $R$ is $\tau$-U-refinable and $R$ is a $\tau$-U-$\alpha$-$\beta$-UFR, then $R$ is a $\tau$-U-$\beta$-FFR.
\\
(3) If $R$ is $\tau$-U-refinable and $R$ is a $\tau$-U-$\alpha$-HFR, then $R$ is a $\tau$-U-BFR.
\\
(4) If $R$ is a $\tau$-U-$\beta$-FFR, then $R$ is a $\tau$-U-BFR.
\\
(5) If $R$ is a $\tau$-U-$\beta$-FFR, then $R$ is a $\tau$-U-$\beta$-WFFR.
\\
(6) If $R$ is a $\tau$-U-$\beta$-WFFR, then $R$ is a $\tau$-U-$\alpha$-$\beta$-divisor finite ring.
\\
(7) If $R$ is $\tau$-U-refinable and $R$ is a $\tau$-U-$\alpha$-BFR, then $R$ satisfies $\tau$-U-ACCP.
\\
(8) If $R$ is $\tau$-U-refinable and $R$ satisfies $\tau$-U-ACCP, then $R$ is $\tau$-U-$\alpha$.
\end{theorem}
\begin{proof} (1) This is immediate from definitions.
\\
\indent (2) Let $a\in R$ be a non-unit.  Let $a=\lambda a_1 \cdots a_n \left\lceil b_1 \cdots b_m\right\rceil$ be the unique $\tau$-$\alpha$-U-factorization up to rearrangement and $\beta$.  Given any other $\tau$-U-factorization, we can $\tau$-U-refine each essential $\tau$-U-divisor into a $\tau$-U-$\alpha$-factorization of $a$.  There is a rearrangement of the essential divisors to match up to $\beta$ with $b_i$ for each $1\leq i \leq m$.  Thus the essential divisors in any $\tau$-U-factorization come from some combination of products of $\beta$ of the $m$ $\tau$-U-$\alpha$ essential factors in our original factorization.  Hence there are at most $2^{m}$ possible distinct $\tau$-U-factorizations up to $\beta$, making this a $\tau$-U-$\beta$-FFR as desired.
\\
\indent (3) For a given non-unit $a\in R$, the number of essential divisors in any $\tau$-U-$\alpha$-factorization is the same, say $N$.  We claim this is a bound on the number of essential divisors of any $\tau$-U-factorization.  Suppose there were a $\tau$-U-factorization $a=\lambda a_1 \cdots a_n \left\lceil b_1 \cdots b_m\right\rceil$ with $m > N$.  For every $i$, $b_i$ has a $\tau$-U-$\alpha$-factorization with at least one essential divisor.  Since $R$ is $\tau$-U-refinable, we can $\tau$-U-refine the factorization yielding a $\tau$-U-$\alpha$-factorization of $a$ with at least $m$ $\tau$-U-$\alpha$ essential factors.  This contradicts the assumption that $R$ is a $\tau$-U-$\alpha$-HFR.
\\
\indent (4) Let $R$ be a $\tau$-U-$\beta$-FFR.  Let $a\in R$ be a non-unit.  There are only finitely many $\tau$-U-factorizations of $a$ up to rearrangement and $\beta$ of the essential divisors.  We can simply take the maximum of the number of essential divisors among all of these factorizations.  This is an upper bound for the number of essential divisors in any $\tau$-U-factorization.
\\
\indent (5) Let $R$ be a $\tau$-U-$\beta$-FFR, then for any non-unit $a\in R$.  Let $S$ be the collection of essential divisors in the finite number of representative $\tau$-U-factorizations of $a$ up to $\beta$.  This gives us a finite collection of elements up to $\beta$.  Every essential divisor up to $\beta$ in a $\tau$-U-factorization of $a$ must be among these, so this collection is finite as desired.
\\
\indent (6) If every non-unit $a\in R$ has a finite number of proper essential $\tau$-U divisors, then certainly there are a finite number of essential $\tau$-$\alpha$-U-divisors.
\\
\indent (7) Suppose $R$ is a $\tau$-U-BFR, but $(a_1) \subsetneq (a_2) \subsetneq \cdots$ is a properly ascending chain of principal ideals such that $a_{i+1}$ is an essential factor in some $\tau$-U-factorization of $a_i$, say 
$$a_i=\lambda_i a_{i1} \cdots a_{i{n_i}}\left\lceil a_{i+1} b_{i1} \cdots b_{im_i}\right\rceil$$ 
for each $i$.  Furthermore, $m_i \geq 1$, for each $i$ otherwise we would have $(a_{i+1})=(a_i)$ contrary to our assumption that our chain is properly increasing.  Our assumption that $R$ is $\tau$-U refinable allows us to factor $a_1$ as follows:
$$a_1=\lambda_1 a_{11} \cdots a_{1{n_1}}\left\lceil a_{2} b_{11} \cdots b_{1{m_1}}\right\rceil=$$
$$\lambda_1 \lambda_2 a_{11} \cdots a_{1{n_1}}a_{21} \cdots a_{2{n_2}}\left\lceil a_{3} b_{21} \cdots b_{2{m_2}} b_{11} \cdots b_{1{m_1}}\right\rceil$$
and so on.  At each iteration $i$ we have at least $i+1$ essential factors in our $\tau$-U-factorization.  This contradicts the assumption that $a_1$ should have a bound on the number of essential divisors in any $\tau$-U-factorization.
\\
\indent (8) Let $a_1\in R$ be a non-unit.  If $a_1$ is $\tau$-U-$\alpha$ we are already done, so there must be a non-trivial  $\tau$-U factorization of $a_1$, say:
$$a_1=\lambda_1 a_{11} \cdots a_{1{n_1}}\left\lceil a_{2} b_{11} \cdots b_{1{m_1}}\right\rceil.$$
Now if all of the essential divisors are $\tau$-U-$\alpha$ we are done as we have found a $\tau$-U-$\alpha$-factorization.  After rearranging if necessary, we suppose that $a_2$ is not $\tau$-U-$\alpha$.  Therefore $a_2$ has a non-trivial $\tau$-U-factorization, say:
$$a_2=\lambda_2 a_{21} \cdots a_{2{n_1}}\left\lceil a_{3} b_{21} \cdots b_{2{m_2}}\right\rceil.$$
Because $R$ is $\tau$-U-refinable, this gives us a $\tau$-U-factorization:
$$a_1=\lambda_1\lambda_2 a_{11} \cdots a_{1{n_1}}a_{21} \cdots a_{2{n_2}}\left\lceil a_{3} b_{21} \cdots b_{2{m_2}} b_{1_1} \cdots b_{1{m_1}}\right\rceil$$
which cannot be $\tau$-U-$\alpha$ or else we would be done.  We can continue in this fashion and get an ascending chain of principal ideals 
$$(a_1) \subseteq (a_2) \subseteq \cdots$$
such that $a_{i+1}$ is an essential $\tau$-U-divisor of $a_i$ for each $i$.  
\\
\indent \emph{Claim}: This chain must be properly ascending.  Suppose $(a_i)=(a_{i+1})$ for some $i$.  When we look at $a_i=\lambda_i a_{i1} \cdots a_{i{n_i}}\left\lceil a_{i+1} b_{i1} \cdots b_{i{m_i}}\right\rceil$, we see that $(a_i)=(a_{i+1} b_{i1} \cdots b_{i_{m_i}})$.  But then we could remove any of the $b_{ij}$ for any $1 \leq j \leq m_i$ and still have $(a_i)=(a_{i+1} b_{i1} \cdots \widehat{b_{ij}} \cdots b_{i{m_i}})$ contradicting the fact that the factorization was a $\tau$-U-factorization since $b_{ij}$ is inessential. 
\\
\indent We certainly have  $(a_i)\subseteq (a_{i+1} b_{i1} \cdots \widehat{b_{ij}} \cdots b_{i{m_i}})$.  To see the other containment holds, $(a_i)=(a_{i+1}) \Rightarrow a_{i+1}=a_i r$ for some $r\in R$, and we can simply multiply by $b_{i1} \cdots \widehat{b_{ij}} \cdots b_{i{m_i}}$ on both sides to see that 
$$a_{i+1}b_{i1} \cdots \widehat{b_{ij}} \cdots b_{i{m_i}}=a_i (rb_{i1} \cdots \widehat{b_{ij}} \cdots b_{i{m_i}})$$
showing the other containment.  Proving the claim.
\\
\indent This is a contradiction to the fact that $R$ satisfies $\tau$-U-ACCP, proving we must in finitely many steps arrive at a $\tau$-U-$\alpha$-factorization of $a_1$, proving $R$ is indeed $\tau$-U-$\alpha$ as desired.
\end{proof}
\indent The following diagram summarizes our results from the Theorems \ref{thm: gen rel} and \ref{thm: tau u relations} where $\star$ represents $R$ being strongly associate, and $\dagger$ represents $R$ is $\tau$-U-refinable:

$$
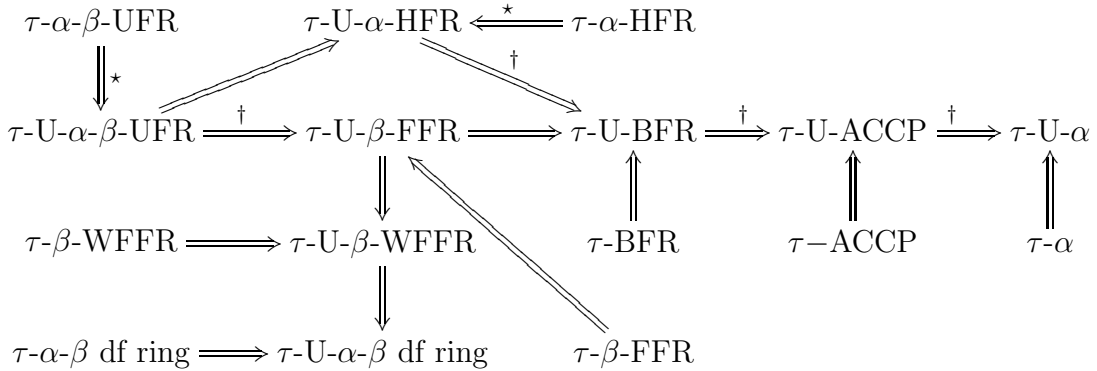
\xymatrix{
    \tau \text{-} \alpha \text{-} \beta \text{-UFR} \ar@{=>}^{\star}[d]       &        \tau \text{-U-} \alpha\text{-HFR} \ar@{=>}^{\dagger}[dr]     &   \tau \text{-} \alpha \text{-HFR} \ar@{=>}_{\star}[l]          &                  &                 \\
\tau \text{-U-} \alpha \text{-} \beta \text{-UFR} \ar@{=>}[ur] \ar@{=>}^{\dagger}[r]  & \tau \text{-U-} \beta \text{-FFR} \ar@{=>}[r] \ar@{=>}[d]  & \tau \text{-U-BFR} \ar@{=>}[r]^{\dagger}& \tau\text{-U-ACCP} \ar@{=>}^{\dagger}[r]& \tau\text{-U-}\alpha\\
 \tau \text{-} \beta \text{-WFFR} \ar@{=>}[r]           & \tau \text{-U-} \beta \text{-WFFR} \ar@{=>}[d]                     &   \tau \text{-} \text{BFR} \ar@{=>}[u]           &  \tau{-} \text{ACCP} \ar@{=>}[u]               &       \tau \text{-} \alpha \ar@{=>}[u]           \\
 \tau\text{-}\alpha \text{-} \beta \text{ df ring}   \ar@{=>}[r]        & \tau\text{-}\text{U-}\alpha \text{-} \beta \text{ df ring} & \tau\text{-}\beta \text{-FFR} \ar@{=>}[uul]
            }$$
\indent We have left off the relations which were proven in \cite[Theorem 4.1]{Mooney}, and focused instead on the rings satisfying the U-finite factorization properties.  Examples given in \cite{Axtell, Axtell2, Frazier, anderson90} show that arrows can neither be reversed nor added to the diagram with a few exceptions. 
\begin{question}Does U-atomic imply atomic?\end{question} 
\indent D.D. Anderson and S. Valdez-Leon show in \cite[Theorem 3.13]{Valdezleon} that if $R$ has a finite number of non-associate irreducibles, then U-atomic and atomic are equivalent.  This remains open in general.  
\\
\begin{question}Does U-ACCP imply ACCP? \end{question}

\indent We can modify M. Axtell's proof of \cite[Theorem 2.9]{Axtell} to add a partial converse to Theorem \ref{thm: tau u relations} (5) if $\tau$ is combinable and associate preserving.  The idea is the same, but slight adjustments are required to adapt it to $\tau$-factorizations and to allow uniqueness up to any type of associate.

\begin{theorem}  Let $\beta \in \{$associate, strongly associate, very strongly associate$\}$.  Let $R$ be a commutative ring with $1$ and let $\tau$ be a symmetric relation on $R^{\#}$ which is both combinable and associate preserving.  $R$ is a $\tau$-U-$\beta$-FFR if and only if $R$ is a $\tau$-U-$\beta$-WFFR.
\end{theorem}
\begin{proof} 
($\Rightarrow$) was already shown, so we need only prove the converse.  ($\Leftarrow$) Suppose $R$ is not a $\tau$-U-$\beta$-FFR.  Let $a \in R$ be a non-unit which has infinitely many $\tau$-U-factorizations up to $\beta$.  Let $b_1, b_2, \ldots, b_m$ be a complete list of essential $\tau$-U-divisors of $a$ up to $\beta$.  Let 
$$a=a_1 \cdots a_n \left \lceil c_1 \cdots c_k \right \rceil =a'_1 \cdots a'_{n'} \left \lceil d_1 \cdots d_n \right \rceil$$
be two $\tau$-U-factorizations of $a$ and assume we have re-ordered the essential divisors in both factorizations above so that the $\beta$ of $b_1$ appear first, followed by $\beta$ of $b_2$, etc.  Let $A=\left \langle (c_1), (c_2), \ldots, (c_k) \right \rangle$ and $B=\left \langle (d_1), (d_2), \ldots, (d_n) \right \rangle$ be sequences of ideals.  We call the factorizations \emph{comparable} if $A$ is a subsequence of $B$ or vice versa.
\\
\indent Suppose $A$ is a proper subsequence of $B$
$$B=\left\langle (d_1), \ldots, (d_{i_1})=(c_1), \ldots, (d_{i_2})=(c_2), \ldots, (d_{i_k})=(c_k), \ldots, (d_n) \right \rangle$$ with $n>k$.  Because $\tau$ is combinable and symmetric, 
$$a=a'_1 \cdots a'_{n'} \left \lceil d_{i_1} d_{i_2} \cdots d_{i_k} (d_1 \cdots \widehat{d_{i_1}} \widehat{d_{i_2}} \cdots \widehat{d_{i_k}} \cdots d_n) \right \rceil$$ 
remains a $\tau$-factorizations and \cite[Lemma 1.3]{Axtell} ensures that this remains a U-factorization.  
\\
\indent This yields
$$ (a)=(d_1 \cdots \widehat{d_{i_1}} \widehat{d_{i_2}} \cdots \widehat{d_{i_k}} \cdots d_n)(d_{i_1} d_{i_2} \cdots d_{i_k}) = (d_1 \cdots d_n)=(c_1 \cdots c_k)$$
$$=(c_1) \cdots (c_k)=(d_{i_1}) \cdots (d_{i_k})=(d_{i_1}\cdots d_{i_k}).$$
But then, $(d_1 \cdots \widehat{d_{i_1}} \widehat{d_{i_2}} \cdots \widehat{d_{i_k}} \cdots d_n)$ cannot be an essential divisor, a contradiction, unless $n=k$.
\\
\indent If $n=k$, then the sequences of ideals are identical, and we seek to prove this means the $\tau$-U-factorizations are the same up to $\beta$.  It is certainly true for $\beta=$ associate as demonstrated in \cite[Theorem 2.9]{Axtell}.  So we have a pairing of the $c_i$ and $d_i$ such that $c_i \sim b_j \sim d_i$ for one of the essential $\tau$-U-divisors $b_j$.  We know further that $c_i$ and $b_j$ (resp. $d_i$ and $b_j$) are $\beta$ since $R$ is by assumption a $\tau$-U-$\beta$-WFFR.  
\\
\indent It is well established that $\beta$ is transitive, so we can conclude that this same pairing demonstrates that $c_i$ and $d_i$ are $\beta$, not just associate.  Thus the number of distinct $\tau$-U-factorizations up to $\beta$ is less than or equal to the number of non-comparable finite sequences of elements from the set $\{ (b_1), (b_2), \ldots, (b_m) \}$.
\\
\indent From here we direct the reader to the proof of the second claim in \cite[Theorem 2.9]{Axtell} where it is shown that this set is finite.

\end{proof} 
\section{Direct Products}
\indent For each $i$, $1\leq i \leq N$, let $R_i$ be commutative rings with $\tau_i$ a symmetric relation on $R_i^{\#}$.  We define a relation $\tau_\times$ on $R=R_1 \times \cdots \times R_N$ which preserves many of the theorems about direct products from \cite{Valdezleon3} for $\tau$-factorizations.  Let $(a_i),(b_i) \in R^{\#}$, then $(a_i) \tau_\times (b_i)$ if and only if whenever $a_i$ and $b_i$ are both non-units in $R_i$, then $a_i \tau_i b_i$.  
\\
\indent For convenience we will adopt the following notation: Suppose $x\in R_i$, then $x^{(i)}=(1_{R_1}, \cdots, 1_{R_{i-1}}, x, 1_{R_{i+1}}, \cdots 1_{R_N})$. so $x$ appears in the $i^{\text{th}}$ coordinate, and all other entries are the identity.  Thus for any $(a_i)\in R$, we have $(a_i)=a_1^{(1)}a_2^{(2)}\cdots a_n^{(n)}$ is a $\tau_\times$-factorization.  We will always move any $\tau_\times$-factors which may become units in this process to the front and collect them there.

\begin{lemma} \label{lem: prod rel} Let $R=R_1 \times \cdots \times R_N$ for $N \in \N$.  Then $(a_i) \sim (b_i)$ (resp. $(a_i) \approx (b_i)$) if and only if $a_i \sim b_i$ (resp. $a_i \approx b_i$) for every $i$.  Furthermore, $(a_i) \cong (b_i)$ implies $a_i \cong b_i$ for all $i$, and for $a_i,b_i$ all non-zero, $a_i \cong b_i$ for all $i$ $\Rightarrow (a_i) \cong (b_i)$.
\end{lemma}
\begin{proof} See \cite[Theorem 2.15]{Valdezleon}.
\end{proof}

\begin{example}If $a_{i_0}=0$ for even one index $1 \leq i_0 \leq N$, then $a_i \cong b_i$ for all $i$ need not imply $(a_i) \cong (b_i)$. \end{example}
Consider the ring $R=\Z \times \Z$, with $\tau_i=\Z^{\#}\times \Z^{\#}$ for $i=1,2$, the usual factorization.  We have $1\cong 1$ and $0 \cong 0$ since $\Z$ is a domain; however $(0,1)=(0,1)(0,1)$ shows $(0,1) \not \cong (0,1)$.

\begin{lemma}\label{lem: one nonunit} Let $R=R_1 \times \cdots \times R_N$ for $N \in \N$ with $\tau_i$ a symmetric relation on $R_i^{\#}$ for each $i$.  Let $\alpha \in \{$ irreducible, strongly irreducible, m-irreducible, very strongly irreducible$\}$.  If $(a_i) \in R$ is $\tau$-$\alpha$, then precisely one coordinate is not a unit.
\end{lemma}

\begin{proof} Let $a=(a_i)\in R$ be a non-unit which is $\tau_\times$-$\alpha$.  Certainly not all coordinates can be units, or else $a \in U(R)$. Suppose for a moment there were at least two coordinates for which $a_i$ is not a unit in $R_i$.  After reordering, we may assume $a_1$ and $a_2$ are not units.  Then $a=a_1^{(1)}(1_{R_1}, a_2, \cdots , a_N)$ is a $\tau_\times$-factorization.  But $a$ is not even associate to either $\tau_\times$-factor, a contradiction.
\end{proof}
\begin{theorem} \label{thm: dir. prod.} Let $R=R_1 \times \cdots \times R_N$ for $N \in \N$ with $\tau_i$ a symmetric relation on $R_i^{\#}$ for each $i$.  
\\
(1) A non-unit $(a_i) \in R$ is $\tau_\times$-atomic (resp. strongly atomic) if and only if $a_{i_0}$ is $\tau_{i_0}$-atomic (resp. strongly atomic) for some $1\leq i_0 \leq n$ and $a_i \in U(R_i)$ for all $i \neq i_0$.
\\
(2) A non-unit $(a_i) \in R$ is $\tau_\times$-m-atomic if and only if $a_{i_0}$ is $\tau_{i_0}$-m-atomic for some $1\leq i_0 \leq n$ and $a_i \in U(R_i)$ for all $i \neq i_0$.
\\
(3) A non-unit $(a_i) \in R$ is $\tau_\times$-very strongly atomic if and only if $a_{i_0}$ is $\tau_{i_0}$-very strongly atomic and non-zero for some $1\leq i_0 \leq n$ and $a_i \in U(R_i)$ for all $i \neq i_0$.
\end{theorem}
\begin{proof} (1) ($\Rightarrow$) Let $a=(a_i) \in R$ be a non-unit which is $\tau_\times$-atomic (resp. strongly atomic).  By Lemma \ref{lem: one nonunit}, there is only one non-unit coordinate.  Suppose after reordering if necessary that $a_1$ is the non-unit.  If $a_1$ were not $\tau_1$-atomic (resp. strongly atomic), then there is a $\tau_1$-factorization, $\lambda_{1_1}a_{1_1}a_{1_2}\cdots a_{1_k}$ for which $a_1 \not \sim a_{1_j}$ (resp. $a_1 \not \approx a_{1_j}$) for any $1 \leq j \leq k$.  But then
$$(a_i)=(\lambda_{1_1},a_2, \ldots, a_n)a_{1_1}^{(1)}a_{1_2}^{(1)}\cdots a_{1_k}^{(1)}$$ 
is a $\tau_\times$-factorization.  Furthermore, by Lemma \ref{lem: prod rel} $(a_i) \not \sim a_{1_j}^{(1)}$ (resp. $(a_i) \not \sim a_{1_j}^{(1)}$) for all $1 \leq j \leq k$.  This would contradict the assumption that $a$ was $\tau_\times$-atomic (resp. strongly atomic).
\\
\indent ($\Leftarrow$) Let $a_1\in R_1$, a non-unit with $a_1$ being $\tau_1$-atomic (resp. strongly atomic).  Let $\mu_i \in U(R_i)$ for $2 \leq i \leq N$.  We show $a=(a_1, \mu_2, \cdots \mu_N)$ is $\tau_\times$-atomic (resp. strongly atomic).  Suppose $a=(\lambda_1, \ldots, \lambda_N)(a_{1_1}, \ldots, a_{1_N})\cdots (a_{k_1}, \ldots, a_{k_N})$ is a $\tau_\times$-factorization of $a$.  We first note $a_{i_j}\in U(R_j)$ for all $j\geq 2$.  Furthermore, this means $a_{i_1}$ is not a unit in $R_1$ for $1\leq i \leq k$, otherwise we would have units as factors in a $\tau_\times$ factorization.  This means $a_1=\lambda_1 a_{1_1} \cdots a_{k_1}$ is a $\tau_1$ factorization of a $\tau_1$-atomic (resp. strongly atomic) element.  Thus, we must have $a_1 \sim a_{j_1}$ (resp. $a_1 \approx a_{j_1}$) for some $1\leq j \leq k$.  Hence by Lemma \ref{lem: prod rel}, we have $a \sim (a_{j_1}, \ldots, a_{j_N})$ (resp. $a \approx (a_{j_1}, \ldots, a_{j_N})$ for some $1\leq j \leq k$ and $a$ is $\tau_\times$ atomic (resp. strongly atomic) as desired.
\\
\indent (2) ($\Rightarrow$) Let $a=(a_i) \in R$ be a non-unit which is $\tau_\times$-m-atomic.  By Lemma \ref{lem: one nonunit}, there is only one non-unit coordinate, say $a_1$ after reordering if necessary.  Let $a_1=\lambda_{1_1}a_{1_1}a_{1_2}\cdots a_{1_k}$ be a $\tau_1$ factorization for which $a_1 \not \sim a_{1_{j_0}}$ for at least one $1 \leq j_0 \leq k$.  But then
$$(a_i)=(\lambda_{1_1},a_2, \ldots, a_n)a_{1_1}^{(1)}a_{1_2}^{(1)}\cdots a_{1_k}^{(1)}$$ 
is a $\tau_\times$-factorization of $a$ for which (by Lemma \ref{lem: prod rel}) $a=(a_i) \not \sim a_{1_{j_0}}^{(1)}$.  This contradicts the hypothesis that $a$ is $\tau_\times$-m-atomic.
\\
\indent ($\Leftarrow$) Let $a_1\in R_1$, a non-unit with $a_1$ being $\tau_1$-m-atomic.  Let $\mu_i \in U(R_i)$ for $2 \leq i \leq N$.  We show $a=(a_1, \mu_2, \cdots \mu_N)$ is $\tau_\times$-m-atomic.  Suppose 
$$a=(\lambda_1, \ldots, \lambda_N)(a_{1_1}, \ldots, a_{1_N})\cdots (a_{k_1}, \ldots, a_{k_N})$$
 is a $\tau_\times$-factorization of $a$.  We first note $a_{i_j}\in U(R_j)$ for all $j\geq 2$.  As before, this means $a_1=\lambda_1 a_{1_1} \cdots a_{k_1}$ is a $\tau_1$ factorization of a $\tau_1$-m-atomic element.  Hence $a_1 \sim a_{j_1}$ for each $1\leq j \leq k$.  By Lemma \ref{lem: prod rel} we have $a \sim (a_{j_1}, \ldots, a_{j_N})$ for all $1\leq j \leq k$ and thus $a$ is $\tau_\times$-m-atomic as desired.
\\
\indent (3) ($\Rightarrow$) Let $a=(a_1, \ldots a_N)$ be a non-unit which is $\tau_\times$-very strongly atomic.  By Lemma \ref{lem: one nonunit}, we may assume $a_1$ is the non-unit, and $a_j$ is a unit for $j \geq 2$.  We suppose for a moment that $a_1=0_1$.  But then $(0, a_2, \ldots a_N)= (0,1, \ldots 1)\cdot(0, a_2, \ldots a_N)$ shows that $a \not \cong a$, a contradiction.  Lemma \ref{lem: prod rel} shows that if $a \cong a$, then $a_i \cong a_i$ for each $1\leq i \leq N$.  Hence, if $a_1$ were not $\tau_1$-very strongly atomic, then there is a $\tau_1$-factorization, $\lambda_{1_1}a_{1_1}a_{1_2}\cdots a_{1_k}$ for which $a_1 \not \cong a_{1_j}$ for any $1 \leq j \leq k$.  But then
$$(a_i)=(\lambda_{1_1},a_2, \ldots, a_n)a_{1_1}^{(1)}a_{1_2}^{(1)}\cdots a_{1_k}^{(1)}$$ 
is a $\tau_\times$-factorization.  Furthermore, since every coordinate is non-zero, by Lemma \ref{lem: prod rel} $(a_i) \not \cong a_{1_j}^{(1)}$ for all $1 \leq j \leq k$.  This would contradict the assumption that $a$ was $\tau_\times$-very strongly atomic.
\\
\indent ($\Leftarrow$) Let $a_1 \in R_1^{\#}$ be $\tau_1$-very strongly atomic.  Let $\mu_i \in U(R_i)$ for $2 \leq i \leq N$.  We show $a=(a_1, \mu_2, \cdots \mu_N)$ is $\tau_\times$-very strongly atomic.  We first check $a\cong a$.  By definition of $\tau_1$-very strongly atomic, $a_1 \cong a_1$.  Certainly as units, we have $\mu_i \cong \mu_i$ for each $i \geq 2$.  Lastly, all of these are non-zero, so we may apply Lemma \ref{lem: prod rel} to see that $a\cong a$.  Suppose $a=(\lambda_1, \ldots, \lambda_N)(a_{1_1}, \ldots, a_{1_N})\cdots (a_{k_1}, \ldots, a_{k_N})$ is a $\tau_\times$-factorization of $a$.  We first note $a_{i_j}\in U(R_j)$ for all $j\geq 2$.  As before, this means $a_1=\lambda_1 a_{1_1} \cdots a_{k_1}$ is a $\tau_1$ factorization of a $\tau_1$-very strongly atomic element.  Hence $a_1 \cong a_{j_1}$ for some $1\leq j \leq k$.  By Lemma \ref{lem: prod rel} we have $a \cong (a_{j_1}, \ldots, a_{j_N})$ and thus $a$ is $\tau_\times$-very strongly atomic as desired.
\end{proof}

\begin{lemma}\label{lem: dir. prod.}Let $R=R_1 \times \cdots \times R_N$ for $N \in \N$ with $\tau_i$ a symmetric relation on $R_i^{\#}$.  Let $\alpha \in \{$irreducible, strongly irreducible, m-irreducible, very strongly irreducible $\}$.  Then we have the following.
\\
(1) If $a=\lambda a_1 \cdots a_n \left\lceil b_1 \cdots b_m\right\rceil$ is a $\tau_i$-U-$\alpha$-factorization of some non-unit $a\in R_i$, then $a^{(i)}=\lambda^{(i)} a_1^{(i)} \cdots a_n^{(i)} \left\lceil b_1^{(i)} \cdots b_m^{(i)}\right\rceil$ is a $\tau_\times$-U-$\alpha$-factorization.  
\\
(2) Conversely, let $a_{i_0}\in R_{i_0}$ be a non-unit and $\mu_i \in U(R_i)$ for all $i \neq i_0$.  Let 
$$(\mu_1, \mu_2, \ldots, \mu_{i_0-1}, a_{i_0},\mu_{i_0+1} \ldots, \mu_N)=(\lambda_i)(a_{1_i})(a_{2_i})\cdots (a_{n_i})\left\lceil (b_{1_i})(b_{2_i})\cdots (b_{m_i})\right\rceil$$
 be a $\tau_\times$-U-$\alpha$-factorization.  Then 
$$a_{i_0}=\lambda_{i_0} a_{1_{i_0}} \cdots a_{n_{i_0}} \left \lceil b_{1_{i_0}} \cdots b_{_{i_0}}\right \rceil$$
 is a $\tau_{i_0}$-U-$\alpha$-factorization.
\end{lemma}
\begin{proof} (1) Let $a=\lambda a_1 \cdots a_n \left\lceil b_1 \cdots b_m\right\rceil$ be a $\tau_i$-U-$\alpha$-factorization of some non-unit $a\in R_i$.  It is easy to see that $a^{(i)}=\lambda^{(i)} a_1^{(i)} \cdots a_n^{(i)} \left\lceil b_1^{(i)} \cdots b_m^{(i)}\right\rceil$ is a $\tau_\times$-factorization.  Furthermore, $b_j \neq 0$ for all $1 \leq j \leq $ or else it would not be a $\tau_i$-factorization.  Hence by Theorem \ref{thm: dir. prod.} $b_j^{(i)}$ is $\tau_\times$-$\alpha$ for each $1\leq j \leq m$.  Thus it suffices to show that we actually have a U-factorization.  
\\
\indent Since $a=\lambda a_1 \cdots a_n \left\lceil b_1 \cdots b_m\right\rceil$ is a U-factorization, we know $a_k(b_1 \cdots b_m)=(b_1 \cdots b_m)$ for all $1 \leq k \leq n$.  In the other coordinates, we have $(1_{R_j})=(1_{R_j})$ for all $j\neq i$.  Hence, we apply Lemma \ref{lem: prod rel} and see that this implies that $a_k^{(i)}(b_1^{(i)} \cdots b_m^{(i)})=(b_1^{(i)} \cdots b_m^{(i)})$ for all $1 \leq k \leq n$.  Similarly we have $b_j(b_1 \cdots \widehat{b_j} \cdots b_m)\neq (b_1 \cdots \widehat{b_j} \cdots b_m)$ which implies $b_j^{(i)}(b_1^{(i)} \cdots \widehat{b_j^{(i)}} \cdots b_m^{(i)})\neq (b_1^{(i)} \cdots \widehat{b_j^{(i)}} \cdots b_m^{(i)})$, so this is indeed a U-factorization.
\\
\indent (2) Let 
$$(\mu_1, \mu_2, \ldots, \mu_{i_0-1}, a_{i_0},\mu_{i_0+1} \ldots, \mu_N)=(\lambda_i)(a_{1_i})(a_{2_i})\cdots (a_{n_i})\left\lceil (b_{1_i})(b_{2_i})\cdots (b_{m_i})\right\rceil$$
be a $\tau_\times$-U-$\alpha$-factorization.  We note that $a_{j_i}\in U(R_i)$ for all $i \neq i_0$ and all $1\leq j \leq n$ and $b_{j_i}\in U(R_i)$ for all $i \neq i_0$ and all $1\leq j \leq m$ since they divide the unit $\mu_i$.  Next, every coordinate in the $i_0$ place must be a non-unit in $R_{i_0}$ or else this factor would be a unit in $R$ and therefore could not occur as a factor in a $\tau_\times$-factorization.  This tells us that 
$$a_{i_0}=\lambda_{i_0} a_{1_{i_0}} \cdots a_{n_{i_0}} \left \lceil b_{1_{i_0}} \cdots b_{_{i_0}}\right \rceil$$
is a $\tau_{i_0}$-factorization.  Furthermore, $(b_{k_i})$ is assumed to be $\tau_\times$-$\alpha$ for all $1\leq k \leq m$, and the other coordinates are units, so $b_{k_{i_0}}$ is $\tau_{i_0}$-$\alpha$ for all $1\leq k \leq m$ by Theorem \ref{thm: dir. prod.}.  Again, we need only show that 
$$a_{i_0}= \lambda_{i_0}a_{1_{i_0}}a_{2_{i_0}}\cdots a_{n_{i_0}}\left\lceil b_{1_{i_0}}b_{2_{i_0}}\cdots b_{m_{i_0}}\right\rceil$$ 
is a U-factorization.  Since all the coordinates other than $i_0$ are units, we simply apply Lemma \ref{lem: prod rel} and see that we indeed maintain a U-factorization.
\end{proof}

\begin{theorem}\label{thm: dir. prod. atom}Let $R=R_1 \times \cdots \times R_N$ for $N\in \N$ with $\tau_i$ a symmetric relation on $R_i^{\#}$.  Let $\alpha \in \{$irreducible, strongly irreducible, m-irreducible, very strongly irreducible$\}$.  Then $R$ is $\tau_\times$-U-$\alpha$ if and only if $R_i$ is $\tau_i$-U-$\alpha$ for each $1\leq i \leq N$.
\end{theorem}
\begin{proof} ($\Rightarrow$) Let $a \in R_{i_0}$ be a non-unit.  Then $a^{(i_0)}$ is a non-unit in $R$ and therefore has a $\tau_\times$-U-$\alpha$-factorization.  Furthermore, the only possible non-unit factors in this factorization must occur in the $i_0^\text{th}$ coordinate.  Thus as in Lemma \ref{lem: dir. prod.} (2), we have found a $\tau_{i_0}$-U-$\alpha$-factorization of $a$ by taking the product of the $i_{0}^\text{th}$ entries.  This shows $R_{i_0}$ is $\tau_{i_0}$-U-$\alpha$ as desired.
\\
\indent ($\Leftarrow$) Let $a=(a_i) \in R$ be a non-unit.  For each non-unit $a_i \in R_i$, there is a $\tau_i$-U-$\alpha$-factorization of $a_i$, say 
$$a_i=\lambda_i a_{i_1} \cdots a_{i_{n_i}}\left\lceil b_{i_1} \cdots b_{i_{m_i}} \right \rceil.$$
If $a_i \in U(R_i)$, then $a_i^{(i)}\in U(R)$ and we can simply collect these unit factors in the front, so we need not worry about these factors.  This yields a $\tau_\times$-U-$\alpha$-factorization
$$a=(a_i)=\prod_{i=1}^{n}\lambda_i^{(i)}a_{i_1}^{(i)} \cdots a_{i_{n_i}}^{(i)} \left\lceil \prod_{i=0}^{m} b_{i_1}^{(i)} \cdots b_{i_{m_i}}^{(i)} \right \rceil.$$
It is certainly a $\tau_\times$-factorization.  Furthermore, $b_{j_k}\neq 0_j$ for $1 \leq j \leq m$ and $1 \leq k \leq m_j$, so $b_{j_k}^{(j)}$ is $\tau_\times$-$\alpha$ by Theorem \ref{thm: dir. prod.}.  It is also clear from Lemma \ref{lem: dir. prod.} that this is a U-factorization, showing every non-unit in $R$ has a $\tau_\times$-U-$\alpha$-factorization.
\end{proof}
\begin{theorem} \label{thm: dir prod idf}
Let $R=R_1 \times \cdots \times R_N$ for $N \in \N$ with $\tau_i$ a symmetric relation on $R_i^{\#}$.  Let $\alpha \in \{$irreducible, strongly irreducible, m-irreducible, very strongly irreducible$\}$ and let $\beta \in \{$associate, strongly associate, very strongly associate$\}$.  Then $R$ is a $\tau_\times$-U-$\alpha$-$\beta$-df ring if and only if $R_i$ is $\tau_i$-U-$\alpha$-$\beta$-df ring for each $1\leq i \leq N$.
\end{theorem}
\begin{proof}($\Rightarrow$) Let $a\in R_{i_0}$ be a non-unit.  Suppose there were an infinite number of $\tau_{i_0}$-U-$\alpha$ essential divisors of $a$, say $\{ b_j \}_{j=1}^\infty$ none of which are $\beta$.  But then $\{ b_j^{(i_0)} \}_{j=1}^\infty$ yields an infinite set of $\tau_\times$-U-$\alpha$-divisors of $a^{(i_0)}$ by Lemma \ref{lem: dir. prod.}.  Furthermore, none of them are $\beta$ by Lemma \ref{lem: prod rel}.
\\
\indent ($\Leftarrow$) Let $(a_i) \in R$ be a non-unit.  We look at the collection of $\tau_\times$-U-$\alpha$ essential divisors of $(a_i)$.  Each must be of the form $(\lambda_1, \cdots, b_{i_0}, \cdots \lambda_N)$ with $\lambda_i \in U(R_i)$ for each $i$ and with $b_{i_0}$ $\tau_{i_0}$-$\alpha$ for some $1 \leq i_0 \leq N$.  But then $b_{i_0}$ is a $\tau_{i_0}$-$\alpha$ essential divisor of $a_{i_0}$.  For each $i$ between $1$ and $N$, $R_i$ is a $\tau_i$-U-$\alpha$-$\beta$-df ring, so there can be only finitely many $\tau_{i}$-$\alpha$ essential divisors of $a_{i}$ up to $\beta$, say $N(a_i)$.  If $a_i \in R_i$, then we can simply set $N(a_i)=0$ since it is a unit and has no non-trivial $\tau_i$-U-factorizations.  Hence there can be only 
$$N((a_i)):= N(a_1) + N(a_2) + \cdots + N(a_N)= \sum_{i=1}^{N} N(a_i)$$
$\tau_\times$-$\alpha$ essential divisors of $(a_i)$ up to $\beta$.  This proves the claim.
\end{proof}
\begin{corollary} 
Let $\alpha$ and $\beta$ be as in the theorem.  Let $R=R_1 \times \cdots \times R_N$ for $N \in \N$ with $\tau_i$ a symmetric relation on $R_i^{\#}$.  Then $R$ is a $\tau_\times$-U-$\alpha$ $\tau_\times$-U-$\alpha$-$\beta$-df ring if and only if $R_i$ is $\tau_i$-U-$\alpha$ $\tau_i$-U-$\alpha$-$\beta$-df ring for each $1\leq i \leq N$.
\end{corollary}
\begin{proof} This is immediate from Theorem \ref{thm: dir prod idf} and Theorem \ref{thm: dir. prod. atom}.
\end{proof}
\begin{theorem} Let $R=R_1 \times \cdots \times R_N$ for $N \in \N$ with $\tau_i$ a symmetric relation on $R_i^{\#}$.  Then $R$ is a $\tau_\times$-U-BFR if and only if $R_i$ is a $\tau_i$-U-BFR for every $i$.
\end{theorem}
\begin{proof} ($\Rightarrow$) Let $a\in R_{i_0}$ be a non-unit.  Then $a^{(i_0)}$ is a non-unit in $R$, and hence has a bound on the number of essential divisors in any $\tau_\times$-U-factorization, say $N_e(a^{(i_0)})$.  We claim this also bounds the number of essential divisors in any $\tau_{i_0}$-U-factorization of $a$.  Suppose for a moment $a=a_1 \cdots a_n \left \lceil b_1 \cdots b_m \right\rceil$ were a $\tau_{i_0}$-U-factorization with $m> N_e(a^{(i)})$.  But then
$$a=\lambda^{(i_0)}a_1^{(i_0)} \cdots a_n^{(i_0)} \left \lceil b_1^{(i_0)} \cdots b_m^{(i_0)}\right \rceil $$
is a $\tau_\times$-U-factorization with more essential divisors than is allowed, a contradiction.
\\
\indent $(\Leftarrow)$ Let $a=(a_i)\in R$ be a non-unit.  Let $B(a)=\text{max}\{N_e(a_i)\}_{i=1}^{N}$.  Where $N_e(a_i)$ is the number of essential divisors in any $\tau_i$-U-factorization of $a_i$, and will say for $a_i \in U(R_i)$, $N_e(a_i)=0$.  We claim that $B(a)N$ is a bound on the number of essential divisors in any $\tau_\times$-U-factorization of $a$.  Let 
$$(a_i)=(\lambda_i) (a_{1_i})\cdots (a_{n_i}) \left \lceil (b_{1_i})\cdots (b_{m_i}) \right \rceil$$
be a $\tau_\times$-U-factorization.  We can decompose this factorization so that each factor has at most one non-unit entry as follows:
$$(a_i)=\prod_{i=1}^{N}\lambda_i^{(i)}a_{1_i}^{(i)} \cdots \prod_{i=1}^{N}a_{n_i}^{(i)} \prod_{i=1}^{N}b_{1_i}^{(i)} \cdots \prod_{i=1}^{N}b_{m_i}^{(i)}.$$
\indent Some of these factors may indeed be units; however, by allowing a unit factor in the front of every $\tau$-U-factorization, we simply combine all the units into one at the front, and maintain a $\tau_\times$-factorization.  We can always rearrange this to be a $\tau_\times$-U-factorization.  Furthermore, since $a_{j_i}$ is inessential, by Lemma \ref{lem: prod rel} $a_{j_i}^{(i)}$ is inessential.  Only some of the components of the essential divisors could become inessential, for instance if one coordinate were a unit.  At worst when we decompose, $b_{j_i}^{(i)}$ remains an essential divisor for all $1\leq j \leq m$ and for all $1\leq i \leq N$.  But then the product of each of the $i^{\text{th}}$ coordinates gives a $\tau_i$-U-factorization of $a_i$ and thus is bounded by $N_e(a_i)$, so we have $m\leq N_e(a_i) \leq B(a)$ and therefore there are no more than $B(a)N$ essential divisors.  Certainly the original factorization is no longer than the one we constructed through the decomposition, proving the claim and completing the proof.
\end{proof}

\begin{theorem} \label{thm: dir. prod. HFR}Let $R=R_1 \times \cdots \times R_N$ for $N \in \N$ with $\tau_i$ a symmetric relation on $R_i^{\#}$.  Let $\alpha \in \{$irreducible, strongly irreducible, m-irreducible, very strongly irreducible $\}$.  Then $R$ is $\tau_\times$-U-$\alpha$-HFR if and only if $R_i$ is a $\tau_i$-U-$\alpha$-HFR for each $i$.
\end{theorem}
\begin{proof} ($\Rightarrow$) Let $a\in R_{i_0}$ be a non-unit.  We know by Theorem \ref{thm: dir. prod. atom} Then $a^{(i_0)}$ is a non-unit in $R$, and has an $\tau_\times$-U-$\alpha$-factorization.  Suppose there were $\tau_{i_0}$-U-$\alpha$-factorizations of $a$ with different numbers of essential divisors, say:
$$a=\lambda a_1 \cdots a_n \left\lceil b_1 \cdots b_m \right\rceil=\mu c_1 \cdots c_{n'} \left\lceil d_1 \cdots d_{m'} \right\rceil$$
where $m\neq m'$.
By Lemma \ref{lem: dir. prod.} this yields two $\tau_\times$-U-$\alpha$-factorizations:
$$a^{(i_0)}=\lambda^{(i_0)}a_1^{(i_0)} \cdots a_n^{(i_0)} \left\lceil b_1^{(i_0)} \cdots b_n^{(i_0)} \right\rceil=\mu^{(i_0)}c_1^{(i_0)} \cdots c_{n'}^{(i_0)} \left\lceil d_1^{(i_0)} \cdots d_{n'}^{(i_0)} \right\rceil.$$
This contradicts the hypothesis that $R$ is a $R$ is $\tau_\times$-U-$\alpha$-HFR.
\\
\indent ($\Leftarrow$) Let $(a_i)\in R$ be a non-unit.  Suppose we had two $\tau_\times$-U-$\alpha$ factorizations
$$(a_i)=(\lambda_i)(a_{1_i})(a_{2_i})\cdots (a_{n_i})\left\lceil (b_{1_i})(b_{2_i})\cdots (b_{m_i})\right\rceil=(\mu_i)(a'_{1_i})(a'_{2_i})\cdots (a'_{n'_i})\left\lceil (b'_{1_i})(b'_{2_i})\cdots (b'_{m'_i})\right\rceil.$$
For each $i_0$, if $a_{i_0}$ is a non-unit in $R_{i_0}$, then since each $\tau_\times$-$\alpha$ element can only have one coordinate which is not a unit, we can simply collect all the $\tau_\times$-divisors which have the $i_0$ coordinate a non-unit.  This product forms a $\tau_{i_0}$-U-$\alpha$-factorization of $a_{i_0}$ and therefore the number of essential $\tau_\times$-factors with coordinate $i_0$ a non-unit must be the same in the two factorizations.  This is true for each coordinate $i_0$, hence $m=m'$ as desired.
\end{proof}

\begin{theorem} Let $R=R_1 \times \cdots \times R_N$ for $N \in \N$ with $\tau_i$ a symmetric relation on $R_i^{\#}$.  Let $\alpha \in \{$irreducible, strongly irreducible, m-irreducible, very strongly irreducible$\}$ and let $\beta \in \{$associate, strongly associate $\}$.  Then $R$ is $\tau_\times$-U-$\alpha$-$\beta$-UFR if and only if $R_i$ is a $\tau_i$-U-$\alpha$-$\beta$-UFR for each $i$.
\end{theorem}
\begin{proof}We simply apply Lemma \ref{lem: prod rel} to the proof of Theorem \ref{thm: dir. prod. HFR}, to see that the factors can always be rearranged to match associates of the correct type.
\end{proof}

\end{document}